\documentclass{amsart}
\usepackage[latin1]{inputenc}
\usepackage{graphicx}
\usepackage{amsfonts,amsmath,amsthm,amssymb,amscd,euscript}
\usepackage{enumerate}

\newcommand{\R}{\mathbb{R}}

\newcommand{\N}{\mathbb{N}}

\newcommand{\Teich}{Teich\-m\"{u}ller~}

\theoremstyle{plain}

\newtheorem{Lem}{Lemma}[section]

\newtheorem{Cor}{Corollary}[section]
\newtheorem{Prop}{Proposition}[section]
\newtheorem{Thm}{Theorem}[section]
\newtheorem*{Thm*}{Theorem}
\newtheorem*{Prop*}{Proposition}
\newtheorem*{Cor*}{Corollary}
\newtheorem*{Lem*}{Lemma}
\theoremstyle{definition}
\newtheorem{Def}{Definition}[section]

\DeclareMathOperator{\Cat}{Cat}

\begin{document}
\title{Rigidity of flat surfaces under the boundary measure}
\author{Klaus Dankwart}
\thanks{ AMS subject classification: 37E35, 57M50, 30F30, 30F60}
\begin{abstract}
Consider a closed marked flat surface $S$ of genus $g\geq 2$ and area 1 and its universal covering $\tilde{S}$. 
We show that the  measure class of the Hausdorff measure of the Gromov boundary of $\tilde{S}$ uniquely determines $S$.  
\end{abstract}
\maketitle
\section{Introduction}
For a fixed closed topological surface $X$ of genus $g\geq 2$ a \textit{marked surface} $(S,f)$  is an isotopy class of homeomorphisms $f:X \to S$.
If  $S$ is endowed with a length metric the \textit{marked length spectrum} of $(S,f)$ is then the length  $f(\alpha)$ for all free homotopy class of closed curves  $\alpha$ on $X$. Suppose  $(S,f)$ and $(T,g)$ are marked Riemannian surfaces of variable negative curvature.  \cite{Otal1990} and \cite{Croke1990} independently showed that $f\circ g^{ -1}$ is isotopic to  an isometry  if and only if  the marked length spectra of $(S,f)$ and $(T,g)$ are equal.\\
 By the work of \cite{Hamenstadt1992} $f\circ g^{ -1}$  is isotopic to an isometry if and only if the push forward of the Bowen Margulis measure on $T^1S$ is in the measure class of the Bowen Margulis measure on $T^1T$. Denote by $G(S):= \partial \tilde{S} \times\partial \tilde{S} -\triangle $ the set of ordered pairs of distinct boundary points of the ideal boundary of the universal covering $\pi:\tilde{S} \to S $. A locally finite Borel measure on $T^1S$ that is invariant under the geodesic flow is naturally one to one with a current i.e. a locally finite Borel measure on $G(S)$ that is invariant under the group of Deck-transformations. So, the statement can be naturally phrased in terms of geodesic currents. \\
In this note we investigate in which way these results translate to marked flat surfaces  i.e marked  surfaces which are endowed with a locally $\Cat$(0), singular flat metric which comes from a half-translation structure, see Section \ref{sectflatmetric} for details.  Flat surfaces are central in the study of the \Teich space of conformal structures, see i.e. \cite{Abiko1980, HubbaM1979}.
Define the equivalence relation on the set of all marked flat surfaces $(S,f)\sim (T,g)$ if $g\circ f^{-1}$ is isotopic to an isometry. The \textit{\Teich space of flat surfaces} $\mathcal{T} (X)$ is then the set of equivalence classes of marked flat surfaces. $\mathcal{T} (X)$ was already studied, see i.e \cite{Smillie2008, Veech1997}.\\ 
 In a more general setting \cite{Troyanov2007,Troyanov2006,Troyanov1986} studied the deformation space of singular cone metrics. \\
\cite[Theorem 2]{Duchin2010}  showed the analogon to  Otals and Crokes theorem that the marked length spectrum  uniquely determines a point in \Teich space.\\
We are interested in which way  Hamenst\"{a}dts result translates to the \Teich space of flat surfaces. Since the flat metric is not smooth we consider measures on the ideal boundary of the universal covering $\pi:\tilde{X} \to X$.\\
  A point  $(S,f) \in \mathcal{T} (X)$  lifts to an isotopy class of homeomorphisms of the universal coverings  $f:\tilde{X} \to \tilde{S}$ that are invariant under the group of Deck transformations.  As $\tilde{S}$ endowed with the lifted flat metric is a geodesic  Gromov $\delta$-hyperbolic space, there exists a family of Gromov metrics $d_{p,\infty},~p \in \tilde{S}$ on the ideal boundary $\partial\tilde{S}\sim S^1$ and so a family of positive finite Hausdorff measures or Patterson-Sullivan measures $\mu_{\tilde{S},p}$ in same measure class $\mu_{\tilde{S}}$.  Observe that for any two marked flat surfaces  $(S,f),(T,g)$ each representative in the isotopy class $g\circ f^{-1}:\tilde{S} \to \tilde{T}$ is a quasi-isometry which extends to a homeomorphism between the visual boundaries  $g\circ f^{-1}:\partial\tilde{S} \to \partial\tilde{T}$ 
\begin{Thm} \label{Thmmain}
Two points $(S,f),(T,g)~ \in \mathcal{T} (X) $  determine,  up to rescaling,  the same point in the \Teich space of flat surfaces if and only if the push-forward of the measure class $(g\circ f^{-1})_*\mu_{\tilde{S}}$ equals $\mu_{\tilde{T}}$. 
\end{Thm}
The idea is the following. Recall that the marked length spectrum uniquely determines a point in the \Teich space of flat surfaces.
Suppose that $(S,f),(T,g)$ are marked flat surfaces of \textit{volume entropy} one that the push-forward of the measure class $(g\circ f^{-1})_*\mu_{\tilde{S}}$ equals $\mu_{\tilde{T}}$. Assume that there is a free homotopy class $\alpha$  in $S$ whose length is smaller than of its image under the marking change.\\ 
 Fix a representative $h$  in the isotopy class $g\circ f^{-1}$ that is invariant under the group of Deck transformations and fix a base point $p\in \tilde{S}$  in the universal covering of $S$. \\
Following \cite{dankwart2011typical}, one can remove a zero set of boundary points $\tilde{A} \subset \partial \tilde{S}$ so that the set of all rays emanating from $p$ which converges to $\partial \tilde{S}-\tilde{A}$ is of the following structure. 
\begin{itemize}
 \item It is a  tree whose vertices is the set of singularities and the base point $p$. At each vertex there are countably many edges attached.  
\item Let $c$ be a, possibly self intersecting, geodesic segment on the base surface $S$ that connects  singularities. Then each ray in the described tree on the universal covering contains infinitely often a lift of $c$ as a subgeodesic segment.  
\end{itemize}
It is well-known that for each multiple $\alpha^n$ of $\alpha$ there is a geodesic representative  of the form as $c$. \\
We make use of the following statement. Suppose that there is a parametrized geodesic ray $r$ on $\tilde{S}$  starting at $p$ with the following properties.
\begin{enumerate}[1)]
\item $r$  is contained in the above tree. 
 \item The length difference for each subgeodesic segment of $r$ under $h$ is bounded, namely  
 $|d(r(s),r(t))- d(h(r(s)),h(r(t)))|<C,\forall s,t>0$ .   
\end{enumerate}
 Then the length of the free homotopy class $\alpha$ on $(S,f)$ is not shorter than on  $(T,g)$.\\
 The proof is as follows.  
Denote by $\alpha^n$ the free homotopy class of a  multiple of $\alpha$.  Let $c$ be a subgeodesic segment  of $ r$ that is a lift of a  geodesic representative of $\alpha^n$. Observe that the length of $c$ is the length of $\alpha^n$. The minimal length in the homotopy class with fixed endpoints of the image of $c$ under the marking change $h$ gives a lower bound of the length of the image of the free homotopy $\alpha^n$ under the marking change.  So the length of  $\alpha^n$ in $S$ is bounded from below by the length of the image of  $\alpha^n$ under the marking change  minus the constant $C$ of the second condition. But since the length of a multiple $\alpha^n$ is $n$ times the length of   $\alpha$ this leads to the statement.\\
So, to show the Theorem we have to find a geodesic ray with the properties described above. One shows
 that the following condition is equivalent to the second condition: \\
\textbf{Condition 2*}: The ray $r$ satisfies $|d(p,r(t))- d(h(p),h(r(t)))|<C$ for all $t$ larger than some some threshold $t_0$.\\    
The set of boundary points $\tilde{A}$ is of measure zero. So it suffices to show that the set of boundary points to which converge a ray that satisfies Condition 2* are of full measure. \\   
This can be done by using \textit {shadows} of balls of fixed radius with respect to the base point $p$. Observe first that similarly to \cite{dankwart2011typical} one can show that
\begin{itemize}
 \item The $\mu_{\tilde{S}}$-measure of a shadow of a ball with center $x$ equals, up to a multiplicative constant, $exp(-d(p,x))$. 
\item The $\mu_{\tilde{T}}$-measure  of the image of the shadow under  $h$ equals $exp(-d(h(p),h(x))$ up to a multiplicative constant.
\end{itemize}
 In both cases one makes use of the fact that the volume entropy of $S$ and $T$ is one.\\
For a ray $r$ consider the shadows of open balls of fixed radius which are centered on the ray at some point $r(t)$. 
So, if for some $t_0>0$ and for all $t>t_0$ the quotient of the measure of the shadow centered at $r(t)$ for $\mu_{\tilde{S}}$ and $h^{-1}_*\mu_{\tilde{T}}$ is bounded by a constant $C_1(r)>0$  then the rays satisfies condition 2*.\\ 
But, by  the Lebesgue Differentiation Theorem this condition is satisfied for all rays that converge to boundary points outside a zero set as the set of shadows forms a Vitali covering, compare Definition \ref{DefVITALI}. \\
The paper is organized as follows. In Section \ref{sectprel}  we recall some abstract facts in metric measure theory. Then we state the main results about $\delta$-hyperbolic spaces and flat surfaces needed later on. In Section \ref{sectmainthem} we show a  technical result which, together with the previous material, shows the main theorem.\\
\textbf{Acknowledgement} This question was raised by Mrs. Hamenst\"{a}dt. I am very grateful for her patience, her support and many helpful discussions.  
\section{Preliminaries about Gromov hyperbolic spaces and flat surfaces}\label{sectprel}
\subsection{Some metric measure theory}
We recall  some standard facts about metric measure theory and refer to \cite{ Federer1969}. \\
Suppose $(X,d)$ is a proper metric space. For a subset $U$,  the diameter $diam(U)$ is then the supremum of distance of points in $U$.
We recall the definition of the \textit{Hausdorff measure} and \textit{Hausdorff dimension}.\\
Let $a \geq 0$ be some number. For any subset $U$ of $X$, $\mu_{\epsilon}^a(U)$ is then the infimum of $\sum\limits_{i\in \N} diam(U_i)^a$ for all countable coverings $U_i,~i \in \N$ of $U$, so that $diam (U_i)\leq \epsilon$. The \textit{Hausdorff measure} of dimension $a\geq 0$ is then $\mu^a:=\lim\limits_{ \epsilon \rightarrow 0} \mu_{\epsilon}^a$ which might be constant zero or locally infinite.\\   
The \textit{Hausdorff dimension} of  $(X,d)$ is then the infimum of all $a$ so that $\mu^a(X)=0$.\\
If the Hausdorff dimension is finite, the \textit{Hausdorff measure} $\mu$ is then the Hausdorff measure whose dimension is the Hausdorff dimension.\\
Next we recall the notion of doubling measures and Vitali coverings. Vitali coverings can also be defined more generally. That special definition is consistent to the general one, compare \cite[Theorem 2.8.17]{ Federer1969} where the function $\delta$ is here defined as the  diameter.\\ 
Denote by $ B(x,r)$ a metric ball of radius $r$ and and center $x$. A locally finite Borel measure  $\nu$ on $(X,d)$ is \textit{doubling} if there is a constant $C>0$ so that 
$$\nu(B(x,2r))\leq C \nu(B(x,r))~ \forall x,r.$$ 
\begin{Def}\label{DefVITALI}
A collection $\mathcal{V}$ of closed subsets of metric space $(X,d)$  with a doubling measure $\nu$ is a \textit{Vitali covering} if:
\begin{itemize}
\item There is a constant $C_{\mathcal{V}}>0$ so that each $U\in \mathcal{V}$ is contained in a metric ball $B$ with $\nu(U)\geq C_{\mathcal{V}} \nu(B)$
\item For each $x \in X$ there is a sequence of neighborhoods $U_i \in \mathcal{V},~i\in \N$ of $x$  so that $diam(U_i)$ tends to zero. 
\end{itemize}
 \end{Def}
Suppose  $f: X \to \R$ is a measurable function. A point $x \in X$ is a \textit{Lebesgue point} with respect to $\mathcal{V}$ and  $f$ if and only if for each sequence of neighborhoods $U_i \in \mathcal{V}$ of $x$  with $\lim\limits_{i\rightarrow \infty}diam(U_i)=0$  
it follows that
$$\lim\limits_{i\rightarrow \infty}\frac{1}{\nu(U_i)}\int_{U_i}fd\nu=f(x).$$
The main application for Vitali sets is the Lebesgue differentiation Theorem, compare \cite[Theorem 2.9.8]{Federer1969}  
\begin{Thm}
The set of Lebesgue points in $X$ is of full measure.
\end{Thm}

\subsection{Gromov hyperbolic spaces and their boundary}\label{sectgrhyp}
We recall the standard facts about proper $\delta$-hyperbolic spaces, compare \cite[ Chapter III] {BridsH1999}.\\ 
\textbf{Convention:} Any metric space is assumed to be complete, proper and geodesic. Every geodesic segment can be extended to a geodesic line.\\
A metric space $\tilde{S}$ is $\delta$-\textit{hyperbolic} if every geodesic triangle in $\tilde{S}$ with sides $a, b, c$ is $\delta$-\textit{slim}: The side $a$ is contained in the $\delta$-neighborhood of $b \cup c$.\\
A $L$-\textit{quasi-isometry} is a mapping  $g:\tilde{S} \to \tilde{T}$ so that the distance of each point in $\tilde{S}$ to  $g(\tilde{T})$ is uniformly bounded  and so that 
$$(1+L)d_{\tilde{T}}(x,y)+L\geq d_{\tilde{S}}(g(x),g(y))\geq (1+L)^{-1}d_{\tilde{T}}(x,y)-L,~ \forall x,y.$$
 A $L$-\textit{quasi-geodesic} is a mapping  $g$ of a line segment $I$ to $\tilde{S}$ so that 
 $$(1+L)|s-t|+L\geq d(g(s),g(t))\geq (1+L)^{-1}|s-t|-L,~ \forall s,t.$$
$\tilde{S}$ admits a boundary which is defined as follows. Fix a point $p\in \tilde{S}$ and for two points $x,y\in \tilde{S}$ we define the \textit{Gromov product} $(x,y)_p :=\frac{1}{2}(d(x,p)+d(y,p)-d(x,y))$. We call a sequence $x_i$ \textit{admissible} if $(x_i,x_j)_p\rightarrow \infty$. We define two admissible sequences
 $x_i,y_i\subset \tilde{S}$ to be equivalent if $(x_i,y_i)_p \rightarrow \infty$. Since $\tilde{S}$ is hyperbolic, this defines an equivalence relation. The boundary $\partial \tilde{S}$ of $\tilde{S}$ is then the set of equivalence classes.\\ 
The \textit{Gromov product on the boundary} is then 
$$(\eta,\zeta)_p=\sup\{\liminf\limits_{i,j}(x_i,y_j)_p~|~\{x_i\} \in \eta ,~\{y_j\} \in \zeta\} $$
\begin{Prop} \label{Prpgrmetric}
 Let $\tilde{S}$ be a $\delta$-hyperbolic space and let $\delta_{\inf}$ be the infimum of all Gromov hyperbolic constants. 
Moreover, let $\xi$ be defined by $2\delta_{\inf}\cdot \log(\xi)=\log(2)$. There is a constant $\epsilon < 1$ so that for any $p\in \tilde{S}$ there is a metric $d_{p,\infty}$ on $\partial \tilde{S}$ which satisfies:
$$\xi^{-(\eta, \zeta)_p}\geq d_{p,\infty}(\eta, \zeta)\geq (1-\epsilon)\xi^{-(\eta, \zeta)_p} $$ 
\end{Prop}
$d_{p,\infty}$ is a \textit{Gromov metric} and $(\partial \tilde{S},d_{p,\infty}) $ the \textit{Gromov boundary}. A quasi-isometry between Gromov $\delta$-hyperbolic spaces extends to a homeomorphism between the Gromov boundaries. 
Any bi-infinite $L$-quasi-geodesic converges to two distinct boundary points, and between any two distinct boundary points there is a connecting bi-infinite geodesic.
\begin{Lem}\label{lemboundHausdist}
There is a function $H(L,\delta)>0$ such that for any $\delta$-hyperbolic space $\tilde{S}$ and for any two $L$-quasi-geodesics $c,c'$ in $\tilde{S}$ with the same endpoints in $\tilde{S} \cup \partial\tilde{S}$ the $H(L,\delta)$-neighborhood of $c$ contains  $c'$.
\end{Lem}
\begin{proof}
This follows from \cite[III 1.7]{BridsH1999} and \cite[I Proposition 3.2]{CoornP1993}
\end{proof}
\textbf{Notation:} If $\tilde{S}$ is a $\delta$-hyperbolic $\Cat(0)$ space denote by $[x,y]_{\tilde{S}},~x,y\in \tilde{S} \cup \partial\tilde{S}$, a parametrized geodesic segment connecting $x$ with $y$. If the space $\tilde{S}$ is clear from the context we abbreviate $[x,y]:=[x,y]_{\tilde{S}}$. For $x\in \tilde{S},~y\in \tilde{S}\cup \partial \tilde{S}$, $[x,y]$ is unique up to reparametrization. \\
For a base point $p\in \tilde{S}$ and  $U\subset \tilde{S}$ the \textit{boundary shadow} $sh_{p,\tilde{S}}(U) \subset \partial\tilde{S}$ is the set of all points $\eta \in \partial \tilde{S}$ such that at least one geodesic ray $[p,\eta]$ connecting $p$ and $\eta$ intersects $U$. Again if the corresponding space is clear from the context  we abbreviate $sh_{p}(U):= sh_{p,\tilde{S}}(U)$.  $sh_{p}(U)$ is Borel in $\partial\tilde{S}$ if and only if $U$ is Borel. \\
 In a $\delta$-hyperbolic space $\tilde{S}$ we need to estimate the size of shadows of balls. Recall that by construction of the Gromov metric  $d_{p,\infty}(\eta,\zeta)$ is comparable to $\xi^{-(\eta,\zeta)_p}$.
\begin{Prop}\label{Propshadowsmetrsmall}
For each $r>0$ there is a constant $C_{sh}(r)>0$ so that the following holds.
 For a pair of distinct points $p,x\in \tilde{S}$  extend the geodesic segment $[p,x]$ to a geodesic ray with endpoint $\eta$. Then the shadow $sh_p(B(x,r))$ is contained in the boundary ball of $d_{p,\infty}$-radius $C_{sh}(r)\xi^{-d(p,x)}$ centered at $\eta$.
\end{Prop}
\begin{proof}
 Define first $C_{sh}(r):=\xi^{r}$. For a point $\zeta \in sh_p(B(x,r))$ fix the parametrized geodesic ray $[p,\zeta]$ that connects  $p$  with $\zeta$. Observe first that the Gromov product is increasing along  geodesic rays i.e.  for $s_1\leq s_2,~ t_1\leq t_2$ it follows
$$([p,\eta](s_1),[p,\zeta](t_1))_p\leq ([p,\eta](s_2),[p,\zeta](t_2))_p.$$
As $[p,\zeta]$ intersects $B(x,r)$ at some point $y$, 
$$(\eta,\zeta)_p \geq (x,y)_p\geq 1/2(d(p,x)+(d(p,x)-r) -r)=d(p,x)-r.$$
So
$$d_{p,\infty}(\eta, \zeta)\leq \xi^{-(\eta, \zeta)_p}\leq C_{sh}(r)\xi^{-d(p,x)}$$  
 \end{proof}
\begin{Lem}\label{Lemnestedshadows}
Let $f:\tilde{S} \to \tilde{T}$  be an $L$-quasi-isometry between $\delta$-hyperbolic spaces.\\ 
For $H(\delta,L)$ as in Lemma \ref{lemboundHausdist} and for $r>(H(\delta,L)+L)(1+L)$ define 
\begin{eqnarray*}
r_{s}&:=&r/(1+L)-H(\delta,L)-L>0\\
r_{b}&:=&(1+L)r+L+H(\delta,L) 
\end{eqnarray*}
Then one concludes that  
$$sh_{\tilde{T},f(p)}(B(f(x),r_s))\subset f(sh_{\tilde{S},p}(B(x,r)))\subset sh_{\tilde{T},f(p)}(B(f(x),r_b))$$
for all points $p,x\in \tilde{S}$.
\end{Lem}
\begin{proof} We only show $sh_{\tilde{T},f(p)}(B(f(x),r_s))\subset f(sh_{\tilde{S},p}(B(x,r)))$ the other inequality is analogous. \\ 
Since $f$ maps $\partial \tilde{S}$ on $\partial \tilde{T}  $ homeomorphically each point in $sh_{\tilde{T},f(p)}(B(f(x),r_s))$ can be represented as $f(\zeta)$. We have to show that $d([p, \zeta]_{\tilde{S}},x)<r$. Since $f([p,\zeta]_{\tilde{S}})$ is a $L$-quasi-geodesic,   $[f(p),f(\zeta)]_{\tilde{T}}$ is contained in   $H(\delta,L)$-neighborhood  of $f([p,\zeta]_{\tilde{S}})$. Therefore $f([p,\zeta]_{\tilde{S}})$ intersects a ball with center $f(x)$ and $\tilde{T}$-radius $r_s+H(L,\delta)$ in a point $f(y)$. As $(1+L)^{-1}d(x,y)-L<d(f(x),f(y))<r_s+H(\delta,L)$ we conclude that $d(x,y)\leq  (1+L)(r_s+H(\delta,L)+L)=r$.
\end{proof}
For $\Gamma$ a group of isometries  acting properly discontinuously, freely and cocompactly on $\tilde{S}$, the Hausdorff measure $\mu_p$ of the Gromov metric $d_{p,\infty}$  on $\partial \tilde{S}$ can be computed by using the theory Patterson-Sullivan measures, see \cite[Section 1-3]{Sulli1979}, \cite[Section 4-8]{Coorn1993}:  
Fix a $\Gamma$-invariant non-zero Radon-measure $\ell$  on $\tilde{S}$ . The \textit{volume entropy} is defined as
$$e(\tilde{S},\Gamma):= \limsup\limits_{R \to \infty}\frac{\log\left(\ell\left(B(p,R)\right)\right)}{R}.$$
\textbf{Convention:} By entropy we mean volume entropy. Here $\tilde{S}$ is the isometric universal covering of $S=\tilde{S}/\Gamma$, so we abbreviate $e(S):=e(\tilde{S},\Gamma)$. We always assume that the group $\Gamma$ acts properly discontinuously, freely and cocompactly by isometries on $\tilde{S}$. 
\begin{Thm} \label{thmhdimentro}
Let $\tilde{S}$ be a $\delta$-hyperbolic $\Cat(0)$-space  and let $p\in \tilde{S}$ be a base point. Denote by $\mu_{p}$ the Hausdorff measure of the Gromov boundary with respect $d_{p,\infty}$
\begin{enumerate}
\item  $\mu_{p}$  is a finite measure supported on $\partial \tilde{S}$. 
\item For the constant $\xi$ as in Proposition \ref{Prpgrmetric} the Hausdorff dimension $d$ equals $e(S)/\log(\xi)$
\item Denote by $d$ again the Hausdorff dimension. Then there is a constant $C>0$, so that the measure $\mu_p$ of a each boundary ball $B(\eta,r)$ of $d_{p,\infty}$-radius $r$ can be estimated by:
$C^{-1}r^d \leq \mu_{p}(B(\eta,r)) \leq  Cr^d$      
\end{enumerate}
\end{Thm}
\begin{proof}
 We refer to \cite{Coorn1993}.
\end{proof}

\subsection{Geometry of flat surfaces}\label{sectflatmetric}
We introduce the geometry of closed flat surfaces and  refer to \cite{Minsk1992}, \cite{Streb1984}. \\
\textbf{Convention:} Any closed surface is assumed to be of genus $g \geq 2$.\\ 
A \textit{half-translation structure} $S$ on a closed topological surface $X$ is a choice of charts such that, away from a finite set of points $\Sigma$, the transition functions are half-translations i.e. they are of the form $z \mapsto \pm z+c$. The pull-back of the flat metric in each chart gives a metric on $X-\Sigma$. We require that the metric on $S-\Sigma$ extends to a singular cone $\Cat$(0)-metric on $S$. Then $S$ is a \textit{flat surface}. 
\begin{Prop}
 In any homotopy class of arcs with fixed endpoints on a closed flat surface there exists a unique local geodesic representative which is length-minimizing.\\ 
In any free homotopy class of closed curves there is a length-minimizing locally geodesic representative which passes through singularities.
\end{Prop}
\begin{proof}
In both cases, the existence of a length minimizing representative follows from a standard Arzel\`{a}-Ascoli argument.\\ 
The uniqueness  of geodesic arcs follows from the absence of geodesic bigons in $\Cat$(0)-spaces.\\
Suppose  $\alpha$ is a length minimizing closed curve that is disjoint from singularities. The half-translation structure on $S-\Sigma$ preserves direction. So, $\alpha$ does not change direction and is therefore simple. It sweeps out an isometrically embedded flat cylinder. Choose this embedding to be maximal. The boundary contains singularities and $\alpha$ can be homotoped to a boundary component without changing length.     
\end{proof}
Consider the isometric universal covering $\pi: \tilde{S} \to S$ with $\Gamma$ the group of deck transformations.  By the \v{S}varc- Milnor Lemma, any two $\Gamma$-invariant length metrics are quasi-isometric  and therefore  $\tilde{S}$ is Gromov $\delta$-hyperbolic.\\ 
 To study geodesic rays in $ \tilde{S}$,  recall that $\mu_{p}, ~p \in \tilde{S}$ is the Hausdorff measure on the Gromov boundary $\partial \tilde{S}$.
We need to estimate the size of shadows of metric balls.\\
Denote by $\Sigma$ the set of singularities on $S$. 
\begin{Prop}\label{PropShadSing}
There is some $C(r)$ so that for each $p \in \tilde{S}$, the $\mu_{p}$-measure of a shadow $sh_p(B(x,r))$ of a ball of radius $r>\sup\limits_{x\in S}d(x,\Sigma)$ centered at $x$ can be estimated by:
$$C(r)^{-1} \exp(-d(p,x))\leq \mu_{p}(sh_{p}(B(x,r)))\leq C(r) \exp(-d(p,x))$$ 
\end{Prop}
\begin{proof}
By Proposition \ref{Propshadowsmetrsmall} and Theorem \ref{thmhdimentro} there is some $C_1(r)>0$ so that $\mu_{p}(sh_{p}(B(x,r))\leq C_1(r) \exp(-d(p,x))$.\\
To show the lower bound observe that $B(x,r)$ contains a singularity $\varsigma$. By \cite{dankwart2011typical} there is a constant $C_2$
so that 
$$C_2^{-1} \exp(-(d(p,x)+r))\leq C_2^{-1} \exp(-d(p,\varsigma))\leq \mu_{p}(sh_{p}(\varsigma))$$ 
Define $C:=\max\{C_1(r),C_2\exp(r)\}$
\end{proof}
\begin{Cor}
Fix a point $p$ on the universal covering $\pi:\tilde{S} \to S$ of a flat surface $S$. Then
the Hausdorff measure $\mu_p$ on $\partial \tilde{S}$ is doubling with respect to the Gromov metric $d_{p,\infty}$. For fixed $r> \sup\limits_{x\in S}d(x,\Sigma_S)$ the collection of shadows $\{sh_p(B(x,r)),~ x\in \tilde{S},r>0\}$ forms a Vitali covering. 
\end{Cor}
\begin{proof}
The fact that $\mu_p$ is doubling follows from Theorem \ref{thmhdimentro}. One concludes from Proposition \ref{Propshadowsmetrsmall} and Proposition \ref{PropShadSing}  that the set of shadows forms a Vitali covering.
\end{proof}
Finally we reformalize a statement from \cite{dankwart2011typical} in a weaker form.  
On a flat surface $S$ of volume entropy one fix a geodesic segment $c$ that starts and ends at singularities.   In the later context, $c$ is a multiple of a closed geodesic that passes through singularities. Fix a base  point $p$ on the universal covering $\pi:\tilde{S} \to S$. Denote by $\tilde{A}_c$  the set of boundary points to which tends a geodesic ray emanating from $p$  which contains at most finitely many lifts of $c$ as a subgeodesic segment. \\
Denote by $\tilde{A} $ the union of  $\tilde{A}_c$ for all such geodesic segments $c$.
\begin{Prop}\label{Propgengeod}
   $\tilde{A}$ is a zero set in with respect to the measure class $\mu_{\tilde{S}}$  
\end{Prop}
\begin{proof}
Since  $S$ contains only finitely many singularities there are only  countably many geodesic segments $c$ on $S$ that connect singularities. So it suffices to show, that $\tilde{A}_c$ is a zero set for each such geodesic segment $c$. Moreover we can assume that $\tilde{A}_c$ is the set of boundary points to which tends a lift of a ray which never passes through $c$. \\    
We first cite the main known results.
\begin{enumerate}
 \item It was shown in \cite[Proposition  4.7]{dankwart2011typical} that the set of boundary points $\eta$ so that $[p,\eta]$ passes through only finitely many singularities is of measure zero. 
\item By \cite[Proposition 3.2]{Dankwart2011} there are constants $C_l,b>0$, so that each geodesic segment from $p$ to a singularity $\varsigma$ can be extended to a geodesic segment of length less than $d(p,\varsigma)+C_l$ which ends at a singularity $\varsigma'$, passes through $b$ singularities and contains a lift of $c$.
\item It  was shown in \cite[Proposition  4.7]{dankwart2011typical} that  there is a constant $C_s>0$ that the $\mu_p$ measure of a shadow of a singularity $\varsigma$ equals $exp(-d(p,\varsigma))$ up to  the multiplicative constant $C_s$.  
\end{enumerate}
For $n\in \N$ define $\Sigma_n$ the set of all singularities $\varsigma$ so that $[p,\varsigma]$ passes trough $n$ singularities but does not pass through a lift of $c$.\\
Define $\Phi_k:\Sigma_{n+k} \to \Sigma_{n}$ , so that $\Phi_k(\varsigma)$ is the $n$-th singularity in $[p,\varsigma]$.  
Also define $\Psi:\Sigma_{n} \to \Sigma_{n+b}$ , so that $\Psi(\varsigma)$ is the endpoint of the extended geodesic segment as in ii) that contains a lift of $c$.\\ 
By i) the shadow of $\Sigma_n$ covers $\tilde{A}_c$ up to measure zero.  For each singularity $\varsigma_n \in \Sigma_n$ one concludes,
 that the shadow of  $\Phi_b^{-1}(\varsigma_n)$ form a subset of the shadow of $\varsigma_n$ whose complement contains the shadow of $\Psi(\varsigma_n)$. By ii) the distance  of $\varsigma_n$ and $\Psi(\varsigma_n)$ is at most $C_l$. By using the estimate of iii) there is a constant $\lambda<1$ so that 
$$\mu_p (sh_p(\varsigma_n ))- \mu_p (sh_p(\Psi(\varsigma) )) \leq \lambda \mu_p (sh_p(\varsigma_n ))$$
So the measure of the shadow of  $\Sigma_{n+b}$ is at most $\lambda$ times the measure of the shadow of  $\Sigma_{n}$.
Therefore letting $n$ tend to infinity  the measure of $sh_p(\Sigma_{n})$ tends to zero and so $\tilde{A}_c$ is a zero set.     
\end{proof}
 
\section{\Teich space of flat surfaces and the proof of the main Theorem}\label{sectmainthem}
\cite[Theorem 2]{Duchin2010} showed that the marked length spectrum uniquely determines a point in the \Teich space of flat surfaces of area one. The following slight generalization is word by word as the original proof.
\begin{Thm} \label{ThmLengthSpectra}
Two marked flat surfaces $(S,f),(T,g)$ of finite area determine the same point in \Teich space of flat surfaces, if and only if the marked length spectrum  of $(S,f)$ equals the marked length spectrum of $(T,g)$. 
\end{Thm}
Suppose that $S,T$ are flat surfaces of volume entropy one and suppose that $f:\tilde{S} \to \tilde{T}$ is a $L$-quasi-isometry between the universal coverings. Fix a base point $p \in \tilde{S}$.  
 \begin{Prop}\label{Propnodist} Let $f:\tilde{S}\to \tilde{T}$ be an $L$-quasi-isometry so that $\mu_{\tilde{S},p}$ is in the measure class of $f_*^{-1}\mu_{\tilde{T},f(p)}$. Then for almost all points $\eta\in \partial \tilde{S}$ there is a constant $C:=C(\eta,h)>0$ that 
 $$\lvert d_{\tilde{S}}([p,\eta]_{\tilde{S}}(s),[p,\eta]_{\tilde{S}}(t))-d_{\tilde{T}}(f([p,\eta]_{\tilde{S}}(s)),f([p,\eta]_{\tilde{S}}(t)) \rvert \leq C, ~ \forall s,t.$$ 
 \end{Prop}
\begin{proof} 
 Define the Radon Nikodym derivative  $h:=\frac{d \left(f_*^{-1}\mu_{\tilde{T},f(p)}\right)}{d\mu_{\tilde{S},p}}$ and denote by $\mathcal{V}$ the  Vitali set of shadows  of balls on $\partial \tilde{S}$ with respect to the base point $f(p)$.  
It will be shown that each Lebesgue point $\eta\in \partial \tilde{S}$ for $h, \mathcal{V}$ with $0<h(\eta)<\infty$  has this property. Choose some $C_1:=C_1(\eta)$ that $C_1^{-1}<h(\eta)<C_1$. 
As $f$ is an $L$-quasi-isometry and and $\tilde{S},\tilde{T}$ are $\delta$-hyperbolic,  one can choose $r,r_b>r_s>0$ as in Lemma \ref{Lemnestedshadows}, so that 
$$r>\sup\limits_{x\in S}(d(x,\Sigma),~ r_s>\sup\limits_{x\in T}(d(x,\Sigma).$$  
  We define first the shadows of balls which are centered on the rays.  
\begin{eqnarray*}
U_{\tilde{S}}(t) &:=&sh_{p,\tilde{S}}(B([p,\eta](t),r))\\
U_{\tilde{T},s}(t) &:=&sh_{p,\tilde{S}}(B(f([p,\eta](t)),r_s))\\
U_{\tilde{T},b}(t) &:=&sh_{p,\tilde{T}}(B(f([p,\eta](t)),r_b))
\end{eqnarray*}
Note that by construction they nested in the following sense.
$$U_{\tilde{T},s}(t)\subset f(U_{\tilde{S}}(t))\subset U_{\tilde{T},b}(t)$$
Observe that the measure of $U_{\tilde{T},s}(t)$ and the measure of $U_{\tilde{T},b}(t)$ differ at most by a multiplicative factor $C_2$.\\
 $\eta$ is a Lebesgue point and the set of shadows forms a Vitali covering. If $t$ is larger than some threshold $t_0$, the $\mu_{\tilde{S},p}$-measure of the shadow $U_{\tilde{S}}(t)$ along the ray $[p,\eta]$ and the $\mu_{\tilde{T},f(p)}$-measure of its image differ at most by the factor $C_1$ as the quotient 
   $$\frac{\mu_{\tilde{T},f(p)}(f(U_{S}(t)))}{\mu_{\tilde{S},p}(U_{S}(t))} =\frac {f_*^{-1}\mu_{\tilde{T},f(p)}(U_{S}(t))}{\mu_{\tilde{S},p}(U_{S}(t))}=\frac {\int_{U_{\tilde{S}}(t)}h d \mu_{\tilde{S},p} }{\mu_{\tilde{S},p}(U_{S}(t))}$$ converges to $h(\eta)$ when  $t$ tends to infinity. \\
So, observe that the $\mu_{\tilde{S},p}$-measure of $U_{S}(t)$ and $\mu_{\tilde{T},p}(U_{T,s}(t))$ differ at most by the multiplicative factor $C_2C_1$ for all $t>t_0$. \\   
   Suppose first that $s=0,~t>t_0$  and define $y:=[p,\eta](t)$.\\ 
Recall that the $\mu_{p,\tilde{S}}$ measure of shadows ball of fixed radius decrease exponentially in their distance to the base point $p$ up to a factor $C_3$.  Up to enlarging $C_3$ the same holds for shadows of balls in $\tilde{T}$.  So   
\begin{eqnarray*}
C_3^{-1} \exp(-d_{\tilde{S}}(p,y))\leq  & \mu_{\tilde{S},p}(U_{S}(t))& \leq C_3 \exp(-(d_{\tilde{S}}(p,y))\\
C_3^{-1} \exp(-d_{\tilde{T}}(f(p),f(y)))\leq  & \mu_{\tilde{T},p}(U_{\tilde{T},s}(t))& \leq C_3\exp(-d_{\tilde{T}}(f(p),f(y)))
\end{eqnarray*} 
One concludes that the distances $d_{\tilde{T}}(f(p),f(y))$ and $d_{\tilde{S}}(p,y)$ differ at most by an additive constant $C_4$.\\      
 If we skip the restriction that $t> t_0$  recall that $f$ is an $L$-quasi-isometry. 
So, for $C_5:=\max\{(1+L)t_0+L, C_4\}$ and $y:=[p,\eta](t)$ one concludes 
$$|d_{\tilde{T}}(f(p),f(y)) - d_{\tilde{S}}(p,y)|\leq  C_5.$$
For $0\leq s\leq t$ denote by $x:=[p,\eta](s),~y:=[p,\eta](t)$.\\
As  $f([p,y]_{\tilde{S}})$ is an $L$-quasi-geodesic, $f(x)$ is in the  $H(L,\delta)$-neighborhood of $[f(p),f(y)]_{\tilde{T}}$. 
So the length of the concatenated geodesics segments $[f(p),f(x)]$ and $[f(x),f(y)]$ and the distance of $f(p)$ and $f(y)$ differ at most by $2H(L,\delta)$. So $d(f(x),f(y))$ and $d(f(p),f(y))-d(f(p),f(x))$ differ at most by the additive constant $2H(L,\delta)$. \\  
Additionally  $d(x,y)$ is precisely $d(p,y)-d(p,x)$. As $d(f(p),f(*))$ and $d(p,*),~*=x,y$ differ at most by an additive error $C_5$ the claim is shown for $C:=2(C_5+H(\delta,L))$.
\end{proof}

\begin{Prop}\label{PropPSmeasdetlength}
Suppose $f:S \to T$ is a  homeomorphism between flat surfaces of volume entropy one and define $f: \tilde{S} \to \tilde{T}$ the lift to the universal coverings.\\   
 If the measure classes $f^{-1}_*\mu_{T}$ and $\mu_{S}$ of the Hausdorff measures on $\partial\tilde{S},\partial\tilde{T}$ are  equal then
the length of each free homotopy class of closed curves $\alpha$ on $S$  satisfies 
$$l_S(\alpha)=l_{T}(f(\alpha)).$$
\end{Prop}
\begin{proof}
Assume on the contrary that up to exchanging the roles of $S$ and $T$
$$l_S(\alpha)<l_{T}(f(\alpha)) - \epsilon$$ 
for some $\epsilon>0$ and for some free homotopy class $\alpha$. 
Fix a base point $p$ on $\tilde{S}$.  By Proposition \ref{Propgengeod} there is a Lebesgue point $\eta \in \partial \tilde{S}$ as in Proposition \ref{Propnodist} so that the projection of $[p,\eta]_{\tilde{S}}$  to $S$ passes through all geodesic segments that connect singularities. Define $C$ as in Proposition \ref{Propnodist} and fix some multiple $\alpha^n$ of $\alpha$ so that $n\geq C/\epsilon$.
Denote by $c$ a geodesic representative of the free homotopy class of  $\alpha^n$ that passes through singularities.  Fix a lift $\tilde{c}$ of $c$ with singularities $\varsigma_1, \varsigma_2$ as endpoints that is a subgeodesic segment of $[p,\eta]_{\tilde{S}}$.
Observe that $d_{\tilde{S}}(\varsigma_1 ,\varsigma_2)=nl_S (\alpha)$. 
On the other hand the projection of the geodesic segment $[f(\varsigma_1),f(\varsigma_2)]_{\tilde{T}}$ is in the free homotopy class of $f(\alpha^n)$ so $d_{\tilde{T}}(f(\varsigma_1),f(\varsigma_2))$ is at least $nl_{T}(f(\alpha))>nl_{S}(\alpha)+C $. 
This contradicts the fact that $d_{\tilde{T}}(f(\varsigma_1),f(\varsigma_2))$ and $d_{\tilde{S}}(\varsigma_1,\varsigma_2)$ differ less than $C$. 
\end{proof}
  \begin{proof}[Proof of Theorem \ref {Thmmain}]
Let $(S,f),(T,g) \in \mathcal{T} (X)$ be points in the \Teich space of flat surfaces  so that the push forward of the measure class $\mu_{\tilde{S}}$ under a lift of an representative in  $g\circ f^{-1}$ equals $\mu_{\tilde{T}}$. Scale $S,T$ to volume entropy one. Then by   
 Proposition \ref{PropPSmeasdetlength} the marked length spectra of $(S,f),(T,g)$ are equal up to a positive scalar. By  Theorem \ref{ThmLengthSpectra} the scaled marked flat surface $(S,f),(T,g)$ determine the same point in $\mathcal{T} (X)$. 
\end{proof}
\bibliographystyle{alpha}	
 \bibliography{datab}
\noindent
Klaus Dankwart\\
Vorgebirgsstrasse 80,\\
53119 BONN, GERMANY\\
e-mail: kdankwart@googlemail.com

\end{document}